\documentclass{amsart}
\usepackage{amsmath}
\usepackage{amssymb}
\usepackage{graphicx}
\input xy
\xyoption{all}
\usepackage{color}
\usepackage{enumerate}
\newtheorem{theorem}{Theorem}[section]

\newtheorem{lemma}[theorem]{Lemma}
\newtheorem{corollary}[theorem]{Corollary}
\newtheorem{proposition}[theorem]{Proposition}
\theoremstyle{definition}
\newtheorem{definition}[theorem]{Definition}

\newcommand{\supp}{\mbox{\rm supp}}

\newcommand{\Aut}{\mbox{\rm Aut}}

\newcommand{\Z}{\mathbb{Z}}

\newcommand{\GL}{{\rm GL}}

\newcommand{\Char}{{\rm char}}

\begin{document}
	
	\title[Locally finite skew group algebras]{On locally finite skew group algebras}
	\author[Bui Xuan Hai]{Bui Xuan Hai $^{1,2}$}
	
	\author[Cao Minh Nam]{Cao Minh Nam$^{1, 2}$}
			
	\author[Mai Hoang Bien]{Mai Hoang Bien$^{1, 2}$}
	\address[1]{Faculty of Mathematics and Computer Science, University of Science, Ho Chi Minh City, Vietnam;}
	\address[2]{Vietnam National University, Ho Chi Minh City, Vietnam.}
	
	\email{bxhai@hcmus.edu.vn; caominhnam.dhsp@gmail.com; mhbien@hcmus.edu.vn }
	
\maketitle
\begin{abstract}
In this note, we study the problem on the existence of non-cyclic free subgroups of the skew group algebra of a locally finite group over a field.  

\end{abstract}	
\section{Introduction and preliminaries} 
 Let $F$ be a field, $G$ a locally finite group, and $F^\sigma G$ the skew group algebra of $G$ over $F$ respect to a group homomorphism $\sigma: G\longrightarrow \Aut(F)$. The aim of this note is to study the existence of non-cyclic free subgroups in an almost subnormal subgroups of $F^\sigma G$. The starting point for the problem on the existence of free sub-objects in algebraic structures is the famous work of Tits \cite[theorems 1, 2]{tits}, where he proved  the existence of non-cyclic free subgroups in linear groups over fields. This result of Tits (which is now often referred to as Tits' Alternative) has been  inspiring various investigations of the existence of free sub-objects in different algebraic structures. In this section, we gather the necessary definitions and some preliminary results we need to establish the main theorem which will be given in the next last section.  

Let $X=\{x_1,\dots,x_m\}$ be a set of noncommuting indeterminates. Then, $X^{-1}$ is understood to be the set $\{x_1^{-1}, \dots, x_m^{-1}\}$ such that $x_i\mapsto x_i^{-1}$ defines a bijective function. For a field $F$, let  $F[X, X^{-1}]$ denote the group algebra over $F$ of the free group on $X$. 
Sometimes, one calls $F[X, X^{-1}]$ the {\it Laurent polynomial ring} on $X$ over $F$. Let $R$ be  a ring whose center contains $F$. Then, the free product $R*_FF[X, X^{-1}]$ of $R$ and $F[X, X^{-1}]$ over $F$ is called the \textit{generalized Laurent polynomial ring} of $R$ in $X$ over $F$ and is denoted by $R_F[X, X^{-1}]$. 
Let $f\in R_F[X, X^{-1}]$, that is, $f(x_1,\dots,x_m)=\sum\limits_{\lambda=1}^{\ell}f_\lambda$  with $$f_\lambda(x_1, \dots,x_m)=a_{\lambda, 1}x_{i_{\lambda,1}}^{n_{\lambda,1}}a_{\lambda, 2}\cdots a_{\lambda, t_\lambda}x_{i_{\lambda,t}}^{n_{\lambda,t}}a_{\lambda,t_\lambda+1},$$ where $a_{i,j}\in R, n_{i,j}\in \Z$. If all $n_{i,j}$ are non-negative, then $f$ is called a {\it generalized polynomial} of $R$ in $X$ over $F$.

An element in $R_F[X, X^{-1}]\backslash R$ of the  form $$w=w(x_1, \dots, x_m)=a_1x_{i_1}^{n_1}\cdots a_tx_{i_t}^{n_t}a_{t+1},$$ where $a_i\in R^*$, $n_i\in \Z$, is called \textit{a generalized group monomial} over $R^*$ (see \cite{Pa_To_85}). 
If all coefficients $a_i$ are $1$, then $w$ is a \textit{group monomial}. We say that $w$ is a \textit{strict generalized group monomial} if whenever $n_jn_{j+1}<0$ for $1\le j< t$, then $i_j\ne i_{j+1}$.


\begin{definition}Let $R_F[X, X^{-1}]$ be as above and $H$ a subgroup of $R^*$. Assume that $w=w(x_1, \dots, x_m)$ is a generalized group monomial over $R^*$.
	If for every $(c_1, \dots,c_m)\in H^m$, there is a positive integer $p$ depending on $(c_1, \dots,c_m)$ such that $[w(c_1, \dots,c_m)]^p\in Z(R)$,  then $w$ is called a \textit{generalized power central group identity} (briefly, GPCGI) of $H$ or $H$ \textit{satisfies the} GPCGI $w$ (see \cite{chiba_88}). Additionally, if $w^p\not\in Z(R)$ for every positive integer $p$, then we say that $w$ is a \textit{non-trivial} GPCGI of $H$.
\end{definition}

\begin{definition} Assume that $w=w(x_1, \dots,x_m)$ is a non-trivial GPCGI over $R^*$ of a subgroup $H\subseteq R^*$. 
	\begin{enumerate}
		\item If $w$ is a group monomial, then we say that $w$ is a \textit{power central group identity} (shortly, PCGI) of $H$.
		\item If $w$ is a strict generalized group monomial, then we say that $w$ is a \textit{strict generalized power central group identity} (shortly, strict GPCGI) of $H$.		
		\item If for every $(c_1, \dots,c_m)\in H^m$, there is a positive integer $p$ depending on $(c_1, \dots,c_m)$ such that $[w(c_1,\dots,c_m)]^p=1$,  then $w$ is called a \textit{generalized torsion group identity} (briefly, GTGI) of $H$. Additionally, if $w$ is a group monomial, then we call $w$ \textit{torsion group identity} (briefly, TGI) of $H$.
		
	\end{enumerate}
\end{definition}

\begin{definition} Let $R_F[X, X^{-1}]$ be as above and $H$ a subgroup of $R^*$. Assume that $w=w(x_1,\dots,x_m)$ is a generalized group monomial over $R^*$.
	We say that $w=1$ is a \textit{generalized group identity} (briefly, GGI) of $H$ if $w(c_1,\dots,c_m)=1$ for any $(c_1, \dots,c_m)\in H^m$. Additionally,
	\begin{enumerate}
		\item If $w$ is a group monomial, then $w=1$ is called a \textit{group identity} (briefly, GI) of $H$.
		\item If $w$ is a strict generalized group monomial, then $w=1$ is called a \textit{strict generalized group identity} (briefly, strict GGI) of $H$.
	\end{enumerate}
\end{definition}

\begin{lemma}\label{2.4}	Let $R$ and $S$ be two $F$-algebras. If $w=w(x_1,\dots,x_m)$ is a strict generalized group monomial (resp., group monomial) over $R$, then so is $w_f$ over $S^*$ for every $F$-algebra homomorphism $f : R\to S$.
\end{lemma}
\begin{proof} We assume that $w=w(x_1,\dots,x_m)=a_1x_{i_1}^{n_1}\cdots a_tx_{i_t}^{n_t}a_{t+1}$ is a strict generalized group monomial over $R^*$, that is, $t\ge 1$ and one has $i_j\ne i_{j+1}$ whenever  $n_jn_{j+1}<0$ for $1\le j<t$. Then, $w$ induces a monomial $$w_f=w_f (x_1,\dots,x_m)=f(a_1) x_{i_1}^{n_1}  \cdots f(a_t)x_{i_t}^{n_t} f(a_{t+1})$$ over $S^*$. It is obvious that $f(a_i)$ is invertible in $S$ for every $1\le i\le t+1$. So,  trivially, $w_f$ is a strict generalized group monomial over $S^*$.
\end{proof}
The following corollary is immediate.
\begin{corollary}\label{2.5} Let $f : R\to S$ be an $F$-algebra homomorphism. If $w$ is a strict GPCGI (resp., strict GGI and GI) of $H$, then  $w_f$ is a strict GPCGI (resp., strict GGI and GI) of $f(H)$ for every subgroup  $H$ of $R^*$.
\end{corollary}	

\section{Free subgroups in skew group algebras}

In this section, firstly, we consider the influence of power central group identities to the structure of  finite dimensional algebras over fields in order to study the subgroup structure of skew group algebras of locally finite groups over fields. As a main result, we prove the theorem (see Theorem \ref{main_new}) on the existence of non-cyclic free subgroups in an almost subnormal subgroup of the skew group algebra of a locally finite group over a field. 

Let $H$ be a group and $N$ a subgroup of $H$. In accordance with Hartley \cite{Pa_Ha_89}, we say that $N$ is \textit{almost subnormal} in $H$  if there is a series of subgroups 
$$N=N_r< N_{r-1}< \dots < N_1=H$$
such that for each $1<i\le r$, either $N_{i}$ is normal in $N_{i-1}$ or the index $[N_{i-1}:N_i]$ is finite. We call such a series of subgroups an \textit{almost normal series} in $H$. Observe that for any group homomorphism $f : H\to K$ and  two subgroups $N_1\le N_2$ of $H$, if $N_1$ is normal in $N_2$ or $[N_2:N_1]<\infty$, then $f(N_1)$ is normal in $f(N_2)$ or $[f(N_2):f(N_1)]<\infty$ respectively. Hence, the following lemma is obvious.

\begin{lemma} \label{3.1} If $N$ is an almost subnormal subgroup in $H$, then $f(H)$ is almost subnormal in $f(H)$ for every group homomorphism $f : H\to K.$
\end{lemma} 

Recall that a field $F$ is {\it absolute} (or \textit{locally finite}) if its prime subfield $P$ is  finite and $F$ is algebraic over $P$. The class of non-absolute fields includes fields of characteristic $0$ and uncountable fields \cite[Corollary B-2.41]{rotman}. Assume that $R$ is an $F$-algebra. If $F$ is absolute, then there is no much information about the structure of $R$. For instance, if $F$ is absolute and $R$ is a finite dimensional $F$-algebra, then every element of $R$ is algebraic over a finite field, which implies that $R^*$ is a torsion group and, hence, $R^*$ always satisfies the power central group identity $x$. Thus, in the following theorem, $F$ is assumed to be non-absolute. 

\begin{theorem}\label{main}
	Let $R$ be a finite dimensional algebra over a non-absolute field $F$ and assume that  $w=w(x_1,\dots,x_m)$ is a strict generalized group monomial over $R^*$. If $N$ is an almost subnormal subgroup of $R^*$ satisfying a strict {\rm GPCGI} $w$, then the derived group $N'=[N,N]$ is contained in $1+J(R)$. Additionally, if $\Char(F)=p>0$, then $N$ satisfies the group identity $(xyx^{-1}y^{-1})^{p^s}=1$ for some positive integer $s$.
\end{theorem} 
\begin{proof} Let $$w(x_1,\dots,x_m)=a_1x_{i_1}^{n_1}\dots a_tx_{i_t}^{n_t}a_{t+1}$$ be a strict GPCGI of $N$.
	Since $R$ is finite dimensional over $F$, by  Wedderburn-Artin's theorem, $$R/J(R)\cong M_{n_1}(D_1)\times\dots\times M_{n_k}(D_k),$$ where $D_i$ is a finite-dimensional division ring over its center $F_i=Z(D_i)$. We claim that all $D_i$ are non-absolute fields. Indeed, without loss of generality, assume that $D_1$ is an absolute field. Then, it is easy to see that $\GL_{n_1}(D_1)$ is a locally finite group and $\Char(D_1)=p>0$. The natural ring homomorphism
	$$\phi: R\longrightarrow R/J(R)$$
	induces the group homomorphism which is also denoted by $\phi$.
	$$\phi: R^*\longrightarrow \GL_{n_1}(D_1)\times\dots\times \GL_{n_k}(D_k).$$
	For each $\alpha\in F^*$, let  $$a=\phi(\alpha)=(a_1,\dots,a_k)\in \GL_{n_1}(D_1)\times\dots\times \GL_{n_k}(D_k).$$ 
	Since $F$ is non-absolute, there exists $\alpha\in F^*$ which is not algebraic over the prime subfield $P\cong \Z_p$ of $F$. Observe that $a_1$  has a finite order in  $\GL_{n_1}(D_1)$, so $a^t=(a_1^t,\dots,a_k^t)=(1,\dots,a_k^t)$ for some positive integer $t$. Hence, $1-a^t=(0,\dots,a_k^t)$ is non-invertible in $R/J(R)$ which implies that $1-\alpha^t$ is non-invertible in $R$. Consequently, $1-\alpha^t$ is non-invertible in $F$, so $\alpha^t=1$, a contradiction. Thus, the claim is proved. 
	
	For each $1\le \ell\le k$, denote by $\phi_\ell \colon \GL_{n_1}(D_1)\times \dots\times \GL_{n_k}(D_k) \to \GL_{n_\ell}(D_\ell)$ the $\ell$-th projection, that is, $\phi_\ell(a_1,\dots,a_k)=a_\ell$ for every $(a_1,\dots,a_k)\in \GL_{n_1}(D_1)\times \dots\times \GL_{n_k}(D_k)$. Put $\psi_\ell=\phi_\ell\phi$ and $M_\ell=\psi(N)$. Then, by Lemma~\ref{3.1}, $M_\ell$ is almost subnormal in $\psi_\ell (R) =\GL_{n_\ell}(D_\ell)$. By Corollary~\ref{2.5},  $$w_{\psi_\ell}(x_1,\dots,x_m)=\psi_\ell (a_{1})x_{i_1}^{\epsilon_1}\psi_\ell(a_{ 2})\dots \psi_\ell(a_{t})x_{i_t}^{\epsilon_t}\psi_\ell(a_{t+1})$$ is a GPCGI of $M_\ell$. In view of \cite[Theorem 2.5]{HBK17}, we have $M_\ell\subseteq F_\ell$. 
	Observe that $$\phi(a)=(\phi_1(a),\dots,\phi_k(a))\in M_1\times \dots\times M_k$$ for every $a\in N,$ so $\phi(N)$ is abelian. On the other hand,  $R^*/(1+J(R))\cong (R/J(R))^*$, hence $\phi(N)=N(1+J(R))/(1+J(R))$. Therefore, $$[N(1+J(R)),N(1+J(R))]\subseteq 1+J(R).$$
	In particular, $N'=[N,N]\subseteq 1+J(R)$, so the first part of the proof is completed. 
	
	Now, assume additionally that $\Char(F)=p>0$. Since $R$ is artinian,  $J(R)$ is nilpotent. Let $s$ be a positive integer such that $J(R)^s=0$. Then, $J(R)^{p^s}=0$ because $p^s>s$. Therefore, for every $a,b\in N$, $$(aba^{-1}b^{-1})^{p^s}-1=(aba^{-1}b^{-1}-1)^{p^s}\in (J(R))^{p^s}=\{0\}.$$
	Thus, $N$ satisfies a group identity $(xyx^{-1}y^{-1})^{p^s}=1$.
\end{proof}

The following result shows the relation between the existence of non-cyclic free subgroups and strict GPCGI's in finite-dimensional $F$-algebras. Note that if $R$ is a finite dimensional algebra over $F$, then $R$ may be considered as an $F$-subalgebra of $M_n(F)$ with $n=\dim_FR$ and the unit group $R^*$ of $R$ as a linear group over $F$.

\begin{lemma}\label{4.1}
	Let $F$ be a field, $R$ a finite dimensional  $F$-algebra, and  $N$ a finitely generated subgroup of the unit group $R^*$ of $R$. Then, the  following assertions are equivalent:
	\begin{enumerate}
		\item $N$ contains no non-cyclic free subgroups.
		\item $N$ satisfies a GI.
		\item $N$ satisfies a strict GGI.
		\item $N$ satisfies a PCGI.
		\item $N$ satisfies a strict GPCGI.
	\end{enumerate}
\end{lemma}
\begin{proof} The implications $(2)\Rightarrow (3)\Rightarrow (5)$ and $(2)\Rightarrow (4)\Rightarrow (5)$ are trivial. It suffices to show $(1)\Rightarrow (2)$ and $(5)\Rightarrow (1)$. Let $n=\dim_FR=n$. Viewing $R^*$ as a subgroup of  $\GL_n(F)$, we conclude that $N$ is a finitely generated linear group over $F$. Suppose that $N=\langle B_1,\dots,B_k\rangle$.
	
	$(1)\Rightarrow (2)$. By Tits' results (\cite[theorems 1, 2]{tits}), $N$ is solvable-by-finite, that is, there exists a  solvable normal subgroup $A$ of $N$ of finite index $m$. Clearly, every solvable group satisfies some group identity, say $w(x_1,\dots,x_t)=1$. Hence,  $N$ satisfies the group identity $w(x_1^m,\dots,x_t^m)=1$.
	
	(5) $\Rightarrow$ (1). Assume that $N$ satisfies a strict GPCGI $$w=w(x_1,\dots,x_m)=A_1x_{i_1}^{n_1}\dots A_tx_{i_t}^{n_t}A_{t+1},$$ where $x_{i_j}\in \{x_1,\dots,x_m\}$,  $n_j\in\Z$, and $A_i\in R^*\leq \GL_n(F)$. Let $P$ be the prime subfield of $F$, $S$ be the subset of $F$ consisting of all entries of all $A_i$ and $B_j$, and let $L=P(S)$ be the subfield of $F$ generated by $S$. It is clear that $N$ is contained in  $\GL_n(L)$ and $N$ satisfies the strict GPCGI $w$ over  $\GL_n(L)$. For every $C_1,\dots,C_m\in N$, there exists a positive integer $\ell=\ell_{(C_1,\dots,C_m)}$ such that $[w(C_1,\dots,C_m)]^\ell\in L$. Hence, according to \cite[Lemma 3.1]{her2}, there exists a fixed positive integer $\alpha$ such that $[w(C_1,\dots,C_m)]^\alpha\in L$ for every $C_1,\dots,C_m\in N$. Hence, $N$ satisfies a strict GGI
	$$v(x_1,\dots,x_m, y_1,\dots,y_m)=$$
	$$[w(x_1,\dots,x_m)]^\alpha [w(y_1,\dots ,y_m)]^\alpha[w(x_1,\dots,x_m)]^{-\alpha} [w(y_1,\dots,y_m)]^{-\alpha}.$$
	In view of  \cite[Theorem 1]{Pa_To_85}, $N$ is solvable-by-finite. Hence, by Tits' results \cite[theorems 1, 2]{tits}, $N$ contains no non-cyclic free subgroups.
\end{proof}

\begin{proposition}\label{prop:4.3} Let $R$ be a ring of characteristic $p> 0$ with the unit group $R^*$. Then, the following assertions hold.
	
	(i) If $a\in R^*$ is an element of infinite order, then $1-a$ is not nilpotent.
	
	(ii) If the Jacobson radical $J(R)$ of $R$ is a nil ideal, then $R^*$ contains a non-cyclic free subgroup if and only if so does the group $(R/J(R))^*$. 
\end{proposition}
\begin{proof} (i) Assume that  $(1-a)^t=0$ for some positive integer $t$. Since $p^t>t$, $0=(1-a)^{p^t}=1-a^{p^t}$. Hence, $a^{p^t}=1$, a contradiction.
	
	(ii) Assume that $N$ is a non-cyclic free subgroup of $R^*$. We shall prove that $\overline{N}=N+J(R)$ is a non-cyclic free subgroup of $(R/J(R))^*$. Assume by contrary that there exist elements $g_1, \dots, g_m\in N$ and a group monomial $f(x_1, \dots, x_m)$ such that $f(\overline{g_1}, \ \dots, \overline{g_m})=\overline{1}$. Then, $1-f(g_1, \dots, g_m)\in J(R)$ which is a nilpotent element. By (i), $f(g_1, \dots, g_m)$ is an element of finite order in $N$ that is a contradiction. 
	
	The converse is obvious because in general, every pre-image of a non-cyclic free subgroup of $(R/J(R))^*$ via the natural homomorphism
	$$\varphi: R\longrightarrow R/J(R)$$
	is a non-cyclic free subgroup of $R^*$. 
\end{proof}
To prove the next result we need also the following lemma.
\begin{lemma}\label{l3.2'}\cite[Lemma 2.1]{Pa_Sa_98}
	If $G$ is a locally finite group and $F$ is a field of characteristic $p>0$, then $G\cap (1+J(FG))=O_p(G)$.
\end{lemma}

Let $F$ be a field, $G$ a group written multiplicatively and $\sigma\colon G\to \Aut(F),g\mapsto \sigma_g $ a group homomorphism. Let 
$F^\sigma G$ denote the skew group algebra of $G$ over $F$. Using the results we get above, now we study the problem on the existence of non-cyclic free subgroups in an arbitrary almost subnormal subgroup of $F^\sigma G$, where $G$ is a locally finite group. Recall that the skew group algebra $F^\sigma G$ of  $G$ over $F$ with respect to $\sigma$ is defined as follows: as a set, $F^\sigma G$ is the vector space over $F$ with the basis $\{g\mid g\in G\}$, that is, every element $x$ of $F^\sigma G$ is written uniquely as a finite sum $x=\sum\limits_{g\in G}a_g g,$ where $ a_g\in F$. The subset $\supp(x)=\{g\in G\mid a_g\ne 0\}$ of $G$ is called the {\it support} of $x$. The addition in $F^\sigma  G$ is defined by 
$$\sum\limits_{g\in G}a_g g+\sum\limits_{g\in G}b_g g=\sum\limits_{g\in G}(a_g+b_g) g.$$ The multiplication is defined as the extension of the following rule: for every $g\in G$ and $a\in F, ga=\sigma_g(a)g$. If $\sigma$ is trivial, that is, $\sigma_g(a)=a$ for every $g\in G, a\in F$, then we write $FG$ instead of $F^\sigma G$. Clearly, $FG$ is \textit{group algebra} of $G$ over $F$. 
If $G$ is a locally finite group, then $F^\sigma G$ is a locally finite $F$-algebra. Now, using results obtained above on subgroups of finite dimensional $F$-algebras, we can prove the following theorem which is the main result in this note.
\begin{theorem}\label{main_new} Let $F$ be a field, $G$ a locally finite group, and $F^\sigma G$ the group algebra of $G$ over $F$ respect to a group homomorphism $\sigma: G\longrightarrow \Aut(F)$. Then, the following statements hold:
	\begin{enumerate}
		\item If $\Char(F)=0$, then every non-abelian almost subnormal subgroup of $(F^\sigma G)^*$ contains a non-cyclic free subgroup.
		\item Let $\Char(F)=p>0$ and $F$ is a non-absolute field. Then
		 
		  $(i)$ Every non-abelian almost subnormal subgroup of $(F^\sigma G)^*$ contains a non-cyclic free subgroup in case either $\sigma$ is non-trivial or the derived subgroup $G'$ of $G$ is not a $p$-group;
		  
		  $(ii)$ $(F^\sigma G)^*$ contains no non-cyclic free subgroups if and only if $\sigma$ is trivial and $G'$ is a $p$-group.
	\end{enumerate}
\end{theorem}
\begin{proof} $(1)$	Note that if $F$ is a field  of characteristic $0$, and $G$ is a locally finite group, then the Jacobson radical $J(F^\sigma G)=0$ (see \cite[Corollary 4.19]{Bo_La_99} and \cite[Theorem~ A]{Pa_Di_03}).  Now, assume that $N$ is a non-abelian almost subnormal subgroup of $(F^\sigma G)^*$. Then, there exist $x=\sum\limits_{g\in G}a_g g, y=\sum\limits_{g\in G}b_g g\in N$ such that $xy\ne yx$. Let $H$ be the subgroup of $G$ generated by $\supp(x)\cup \supp(y)$ and consider the group algebra $F^\sigma H$. Put $M=N\cap (F^\sigma H)^*$. Then, $M$ is an almost subnormal subgroup in $(F^\sigma H)^*$. If $M$ contains no non-cyclic free subgroups, then by Lemma~\ref{4.1}, $M$ satisfies a strict GPCGI. By Theorem~\ref{main}, $M'\subseteq 1+J(F^\sigma H)$. By remark in the beginning of the proof, $J(F^\sigma H)=0$, so $M'=1$, and, then, $M$ is abelian. As a result, $xy=yx$, a contradiction, so $(1)$ holds.
	
$(2)$ $(i)$ Assume that $N$ is a non-abelian almost subnormal subgroup of $F^\sigma G$, and $N$ contains no non-cyclic free subgroups. We have to show that $\sigma$ is trivial and $G'$ is a $p$-group. Indeed, since $N$ is non-abelian, there exist $x=\sum\limits_{u\in G}a_u u$ and $y=\sum\limits_{u\in G}b_u u$ in $N$ such that $xy\ne yx$. Let $z\in G'$, that is, $z=g_1h_1g_1^{-1}h_1^{-1}\dots g_mh_mg_m^{-1}h_m^{-1}$ for some $g_1,\dots,g_m, h_1,\dots,h_m\in G$. Let $g\in G$, and  $H$ be the subgroup of $G$ generated by the set $$\{g_1,\dots,g_m, h_1,\dots,h_m\}\cup \supp(x)\cup \supp(y)\cup \{g\}.$$ It is easy to see that the skew group algebra $F^\sigma  H$ is a subalgebra of $F^\sigma G$. Also, $F^\sigma  H$ is a finite-dimensional vector space over $F$ because $H$ is a finite group.	Put $M=N\cap (F^\sigma H)^*$. It is obvious that $M$ is almost subnormal in $(F^\sigma H)^*$. If $M$ contains no non-cyclic free subgroups, then by Lemma~\ref{4.1}, $M$ satisfies a strict GPCGI. In  view of Theorem~\ref{main}, $M'\subseteq 1+J(F^\sigma H)$. In particular, $z=1+\alpha$, where $\alpha\in J(F^\sigma H)$. Since $F^\sigma H$ is artinian, $J(F^\sigma H)$ is nilpotent. Let $s$ be the nilpotency index of  $J(F^\sigma H)$.  Then, ${z}^{p^s} =(1+\alpha)^{p^s}=1+\alpha^{p^s}=1$, which implies that the order of $z$ is $p^n$ for some non-negative integer $n$. This holds when $z$ ranges over $G'$, so $G'$ is a $p$-group. Moreover, for any $c\in F$, we have $c^{-1}\sigma_g(c)=c^{-1} gc g^{-1}\in M'\subseteq  1+J(F^\sigma H)$. Hence, $1-c^{-1}\sigma_g(c)\in J(F^\sigma H)$. This occurs if and only if $c=\sigma_g(c)$ since $1-c^{-1}\sigma_g(c)\in F$. Again, this holds when $c$ ranges over $F$ and $g$ ranges over $G$, so $\sigma$ is trivial. Thus, $(i)$ holds. 

$(ii)$ The ``If" follows from $(i)$. To show ``Only if", assume that $\sigma$ is trivial and $G'$ is a $p$-group. Then, $F^\sigma G=FG$, and we claim that $(FG/J(FG))^*$ is abelian. It suffices to show $gh-hg\in J(FG)$ for every $g,h\in G$. Indeed, if $g,h\in G$ then $gh=chg$, where $c=ghg^{-1}h^{-1}\in G'\subseteq O_p(G)$. By Lemma~\ref{l3.2'}, $c=1+z$ for some $z\in J(FG)$. Hence, $gh-hg=(1+z)hg-hg=zgh\in J(FG)$. The claim is shown. Hence, $(F G/J(F G))^*$ contains no non-cyclic free subgroups, neither does $(FG)^*$ by Proposition~\ref{prop:4.3}. Hence, $(ii)$ holds.
\end{proof}

%
%
%
\noindent\textbf{Acknowledgement.} The authors would like to express their sincere gratitude
to the referee. This research was funded by Vietnam National Foundation for Science and Technology Development (NAFOSTED) under Grant No. 101.04-2019.323.

\end{document}